\documentclass{amsart}

\usepackage{amsmath,amsthm,amsfonts,amssymb,amscd,latexsym,mathrsfs}
\usepackage[all]{xy}
\usepackage{color}

\newtheorem{thm}{Theorem}[section]
\newtheorem{lem}[thm]{Lemma}
\newtheorem{pro}[thm]{Proposition}
\newtheorem{cor}[thm]{Corollary}
\newtheorem{rem}[thm]{Remark}

\numberwithin{equation}{section}

\newcommand{\dg}{\mathrm{deg}}
\newcommand{\id}{\mathrm{id}}
\newcommand{\Ker}{\mathrm{Ker}}
\newcommand{\im}{\mathrm{Im}}
\newcommand{\Hom}{\mathrm{Hom}}

\newcommand{\ad}{\mathrm{ad}}
\newcommand{\Der}{\mathrm{Der}}
\newcommand{\HCE}{\mathrm{H}}
\newcommand{\dd}{\mathrm{d}}
\newcommand{\stab}{\mathrm{Stab}}

\newcommand{\N}{\mathbb{N}}
\newcommand{\Z}{\mathbb{Z}}

\newcommand{\F}{\mathbb{F}}

\newcommand{\gf}{\mathfrak{g}}

\newcommand{\slf}{\mathfrak{sl}}
\newcommand{\spf}{\mathfrak{sp}}
\newcommand{\psl}{\mathfrak{psl}}

\newcommand{\mc}{\mathcal{M}}
\newcommand{\wc}{\mathcal{W}}
\newcommand{\vc}{\mathcal{V}}

\newcommand{\xb}{\mathbf{x}}
\newcommand{\ab}{\mathbf{a}}
\newcommand{\bb}{\mathbf{b}}
\newcommand{\eb}{\mathbf{e}}
\newcommand{\zero}{\mathbf{0}}

\begin{document}


\title[Simple derivation algebras]{Lie algebras whose derivation algebras are simple}

\author{J\"org Feldvoss}
\address{Department of Mathematics and Statistics, University of South Alabama,
Mobile, AL 36688-0002, USA}
\email{jfeldvoss@southalabama.edu}

\author{Salvatore Siciliano}
\address{Dipartimento di Matematica e Fisica ``Ennio de Giorgi", Universit\`a del
Salento, Via Provinciale Lecce-Arnesano, I-73100 Lecce, Italy}
\email{salvatore.siciliano@unisalento.it}

\dedicatory{}




\begin{abstract}
It is well known that a finite-dimensional Lie algebra over a field of characteristic zero
is simple exactly when its derivation algebra is simple. In this paper we characterize those
Lie algebras of arbitrary dimension over any field that have a simple derivation algebra.
As an application we classify the Lie algebras that have a complete simple derivation
algebra and are either finite-dimensional over an algebraically closed field of prime
characteristic $p>3$ or $\Z$-graded of finite growth over an algebraically closed field
of characteristic zero.
\end{abstract}


\date{January 26, 2025}

\maketitle


\section*{Introduction}


The extent to which the structure of a Lie algebra is determined by its derivation algebra
has already been studied extensively in the past. 

The main result of this paper is a characterization of Lie algebras whose derivation algebras
are simple. For a finite-dimensional Lie algebra $L$ over a field of characteristic zero, Hochschild
showed in \cite[Theorem 4.4]{H} that its derivation algebra $\Der(L)$ is semi-simple if, and
only if, $L$ is semi-simple. As a consequence, we have in this case that $\Der(L)\cong L$,
and $\Der(L)$ is simple exactly when $L$ is simple (see Theorem~3\,(1) and the observation
after the Remark on page 303 in \cite{dR}). These results fail to hold for Lie algebras defined
over fields of non-zero characteristic or for Lie algebras of infinite dimension. Some partial
results for Lie algebras whose derivation algebras are simple were obtained by de Ruiter in
\cite[Theorem 3]{dR}. We settle this problem here in full generality by proving that, for a
Lie algebra $L$ of arbitrary dimension defined over any field, the derivation algebra $\Der(L)$
is simple if, and only if, $L$ is a suitable homomorphic image of the \emph{universal} central
extension of a simple Lie algebra (see Theorem~\ref{simple}). Moreover, we show that
$\Der(L)$ is simple and complete if, and only if, $L$ is a covering algebra of a complete
simple Lie algebra (see Corollary \ref{simplecomplete}). Recall that a Lie algebra is called
\emph{complete} if it has zero center and every derivation is inner, and a \emph{covering
algebra} of a Lie algebra $L$ is a central extension of $L$ that is perfect. As an application,
we classify those Lie algebras that have a complete simple derivation algebra and are
either finite-dimensional over an algebraically closed field of prime characteristic $p>3$ (see
Corollary~\ref{primchar}) or $\Z$-graded of finite growth over an algebraically closed
field of characteristic zero (see Corollary~\ref{char0}).  Since every finite-dimensional
simple Lie algebra over a field of characteristic zero is always complete, Corollaries
\ref{primchar} and \ref{char0} generalize the above mentioned consequence of
Hochschild's classical result to these classes of modular or infinite-dimensional Lie algebras. 

In this paper the identity function on a set $X$ will be denoted by $\id_X$, the additive
group of integers will be denoted by $\Z$, the additive monoid of non-negative integers
will be denoted by $\N_0$, and the additive semigroup of positive integers will be denoted
by $\N$. All other notation is either standard or the same as in \cite{StrI}.


\section{Derivations and central extensions of Cartan algebras}\label{dercartanalg}


Let $r$ be a positive integer, and let $A:=\F[x_1,\dots,x_r]$ denote the algebra of polynomials
in $r$ commuting variables $x_1$, \dots, $x_r$ with coefficients in a field $\F$ of characteristic
zero. In the following we briefly describe the four families of infinite-dimensional Cartan algebras
(see \cite[Section 0.1.3]{M}). Let $W_r:=\Der(A)$ be the Lie algebra of derivations of $A$,
the so-called {\em Witt algebra\/}. Recall that the module of K\"ahler differentials of $A$ is the
free $A$-module $\Omega^1_{A/\F}:=A\dd x_1\oplus\cdots\oplus A\dd x_r$ generated by
the differentials $\dd x_1$, \dots, $\dd x_r$ of the variables of $A$. Note that any derivation
$D\in W_r$ acts in a natural way as a derivation on the exterior algebra $\Omega_{A/\F}:=
\bigoplus\limits_{n\in\N_0}\left(\Lambda^n\Omega^1_{A/\F}\right)$ of K\"ahler differential forms
on $A$. It is straightforward to see that for any differential form $\omega\in\Omega_{A/\F}$
the sets $\{D\in W_r\mid D(\omega)=0\}$ and $\{D\in W_r\mid D(\omega)\in A\omega\}$
are subalgebras of $W_r$.

In this way one obtains the following three families of simple subalgebras of the Witt algebras
(see \cite[Section 0.1.3]{M}):
\begin{itemize}
\item The {\em special Lie algebras\/} $S_r:=\{D\in W_r\mid D(\dd x_1\wedge\cdots\wedge\dd x_r)
         =0\}$,
\item the {\em Hamiltonian algebras\/} $H_{2r}:=\{D\in W_{2r}\mid D\left(\sum\limits_{j=1}^r
         \dd x_j\wedge\dd x_{j+r}\right)=0\}$,
\item the {\em contact Lie algebras\/} $K_{2r+1}:=\{D\in W_{2r+1}\mid D(\omega)\in A\omega\}$,
         where $\omega:=\dd x_{2r+1}+\sum\limits_{j=1}^r(x_j\dd x_{j+r}-x_{j+r}\dd x_j)$.
\end{itemize}         
These four families of infinite-dimensional simple Lie algebras are called {\em Cartan algebras\/}.
It is well known that every Cartan algebra $L$ is simple (see \cite[Lemma~1.14]{M}) and graded
over the integers, i.e., $L=\bigoplus\limits_{n\in\Z}L_n$ such that $[L_m,L_n]\subseteq L_{m+n}$.
In this section we will determine which of the Cartan algebras are complete and then compute
the central extensions of the complete Cartan algebras.

\begin{pro}\label{complete}
Let $L$ be a Cartan algebra over a field $\F$ of characteristic zero. Then $L$ is complete if,
and only if, $L$ is a Witt algebra $W_r$ or a contact Lie algebra $K_{2r+1}$.
\end{pro}

\begin{proof}
As every Cartan algebra $L$ is simple, we have that $C(L)=0$. Hence, it suffices to compute
the first adjoint cohomology space $\HCE^1(L,L)$.

We begin by considering the Witt algebras $W_r$. In order to show that the first adjoint
cohomology $\HCE^1(W_r,W_r)$ vanishes, we will use that $W_r$ can be considered as
a $G$-graded subalgebra of the two-sided Witt algebra
$$W(G,I):=\bigoplus_{\ab\in G}\F w(\ab,i)\,,$$
where $G:=\Z^r$, $I:=\{1,\dots,r\}$, and the Lie bracket on the basis elements $w(\ab,i)$
is defined by
$$[w(\ab,i),w(\bb,j)]=a_jw(\ab+\bb,i)-b_iw(\ab+\bb,j)$$
for all $\ab=(a_1,\ldots,a_r),\bb=(b_1,\ldots,b_r)\in G$ (see \cite[Section~1]{IK}).
Namely, set $\xb^\ab:=x_1^{a_1}\cdots x_r^{a_r}$ for every $r$-tuple $\ab=
(a_1,\dots,a_r)$ of integers, $\eb_j:=(\delta_{ij})_{i\in I}$, where $\delta_{ij}$
denotes the usual Kronecker symbol, and $\partial_i:=\frac{\partial}{\partial x_i}$
for every $i\in I$. Then the linear transformation $\theta:W_r\to W(G,I)$ defined
by $\xb^\ab\partial_i\mapsto -w(\ab-\eb_i,i)$ is a monomorphism of Lie algebras.
Moreover, the Witt algebra $L:=\theta (W_r)$ inherits a $G$-grading $L=
\bigoplus\limits_{\ab\in G}L_\ab$ from $W(G,I)$ via $L_\ab:=W_\ab\cap L$,
where $W_\ab:=\bigoplus\limits_{i=1}^r\F w(\ab,i)$, and one observes that
\begin{eqnarray*}
L_\ab=
\left\{
\begin{array}{cl}
W_\ab & \mbox{ if }a_i\ge 0\mbox{ for every }i\in I\\
\F w(\ab,j) & \mbox{ if }a_j=-1\mbox{ for exactly one }j\in I\mbox{ and }
a_i\ge 0\mbox{ for every }i\in I\setminus\{j\}\,.\\
0 & \mbox{ otherwise}
\end{array}
\right.
\end{eqnarray*}
In particular, we have that $L_\zero=W_\zero$, where $\zero:=(0,\dots,0)\in G$,
and $L_{-\eb_i}=\F w(-\eb_i,i)$ for every $i\in I$.

Next, we would like to apply \cite[Proposition 1.2]{F3}. Note that it follows from
\cite[Lemma 3.11]{M} in conjunction with \cite[Lemma 3.4]{M} that $L$ is finitely
generated unless $r=1$. But in the latter case it is clear that $\{\theta(\partial_1),
\theta(x_1\partial_1),\theta(x_1^2\partial_1),\theta(x_1^3\partial_1)\}$ generates $L$.

It remains to verify that $\HCE^1(L_\zero,L_\ab)=0$ for every $r$-tuple $\ab\ne
\zero$ of integers and $\Hom_{L_\zero}(L_\ab,L_\bb)=0$ for all $\ab=(a_1,\ldots,
a_r),\bb=(b_1,\ldots,b_r)\in G$ with $\ab\ne\bb$. Note that $L_\ab$ is either 0
or $\F w(\ab,j)$ for some integer $j\in I$ or the direct sum of the one-dimensional
$L_\zero$-modules $\F w(\ab,i)$ ($i\in I$). If $\ab\ne\zero$, then we have that
$[\F w(\ab,i)]^{L_\zero}=0$ for every integer $i\in I$, and since $L_\zero$ is abelian,
it follows from \cite[Lemma~3]{Ba} that $\HCE^1(L_\zero,L_\ab)=0$. Moreover,
we obtain that
$$\Hom_{L_\zero}(L_\ab,L_\bb)=\bigoplus_{i\in I_\ab,j\in I_\bb}\Hom_{L_\zero}
(\F w(\ab,i),\F w(\bb,j))$$
for some subsets $I_\ab\subseteq I$ and $I_\bb\subseteq I$, and thus it suffices
to show that every summand vanishes when $\ab\ne\bb$. So by contraposition,
we suppose that there exists a non-zero element $\psi\in\Hom_{L_\zero}(\F
w(\ab,i),\F w(\bb,j))$ for fixed but arbitrary integers $i\in I_\ab$ and $j\in I_\bb$.
Then $\psi(w(\ab,i))=\lambda w(\bb,j)$ for some non-zero scalar $\lambda\in\F$,
and therefore we obtain for every integer $k\in I$ that
$$-a_k\lambda w(\bb,j))=\psi([w(\zero,k),w(\ab,i)])=[w(\zero,k),\psi(w(\ab,i))]=
-b_k\lambda w(\bb,j)\,$$
i.e., $(a_k-b_k)\lambda w(\bb,j)=0$. Hence, we conclude that $\ab=\bb$.

Consequently, we deduce from \cite[Proposition 1.2]{F3} that
$$\Der(L)=\ad(L)+\Der(L)_\zero\,,$$
where $\Der(L)_\zero$ denotes the subspace of the degree-preserving derivations
of $L$, and it remains to show that every derivation $D\in \Der(L)_\zero$ is inner.
For this purpose, we will now use the arguments of Ikeda and Kawamoto in their
proof of Theorem~2 for a finite index set (see \cite[Section 3]{IK}). Note that
the elements $w(-\eb_i,i)$ and $w(\ab+e_i,i)$ in (3.10) and (3.11) are contained
in $L$. Of course, for the latter this might only be the case when the corresponding
element $w(\ab,i)$ belongs to $L$. The same is true for the elements $w(\eb_h,h)$,
$w(\eb_h,i)$ ($h\ne i$), and $w(2\eb_h,i)$ in (3.20)--(3.23). Finally, this also holds
for the element $x$ in (3.31) as $L_\zero=W_\zero=\bigoplus\limits_{i=1}^r\F w
(\zero,i)$. In conclusion, the arguments in Section 3 of \cite{IK} prove that every
degree-preserving derivation of $L$ is inner, and therefore $\Der(L)=\ad(L)$, as
desired.

We now consider the contact Lie algebras $K_{2r+1}$. In the notation of the paper
\cite{SS}, one observes that $K_{2r+1}$ is isomorphic to the Lie algebra $\mathscr{K}
(\vec{\ell},\sigma,\Gamma,\mathscr{I})$ in the particular case when $\vec{\ell}:=
(0,0,0,0,0,r)$, $\sigma:=0$, $\Gamma:=\{0\}$, and $\mathscr{I}:=\mathbb{N}_0^{1+2r}$.
As a consequence, it follows from \cite[Theorem 3.1]{SS} that every derivation of
$K_{2r+1}$ is inner.

In order to show that the special Lie algebras and the Hamiltonian algebras have outer
derivations, we use their usual $\Z$-gradings and consider the ``degree" derivation
$D_0:=(\ad\,h)_{\vert L}\in\Der(L)$ for $L:=S_r$ or $L:=H_{2r}$, where $h:=
\sum\limits_{i=1}^rx_i\partial_i\in W_r$. Suppose that $D_0=\ad_L\,D$ for some
$D\in L$. Then we obtain that
$$\dg(D)D=[h,D]=D_0(D)=[D,D]=0$$
(cf.\ Theorem 2.5\,(1) in Chapter 4 of \cite{SF}), and therefore we deduce that
$\dg(D)=0$ or $D=0$, i.e., $D\in L_0$. In the first case $[D,X]=[h,X]=\dg(X)X=0$
for any $X\in L_0$ also shows that $D\in C(L_0)=0$ (as $L_0\cong\slf_r(\F)$ or
$L_0\cong\spf_{2r}(\F)$ are both simple Lie algebras), i.e., in both cases we
conclude that $D_0=0$, which contradicts $D_0=\id_{L_1}$ and $L_1\ne 0$.
\end{proof}


The next result is a special case of the corresponding results for generalizations
of the Witt and contact Lie algebras obtained by Su and Zhao \cite{SZ} and Song
and Su \cite{SS}, respectively.

\begin{pro}\label{centext}
Let $\F$ be a field of characteristic zero. Then $\HCE^2(W_r,\F)=0$ and
$\HCE^2(K_{2r+1},\F)=0$.
\end{pro}

\begin{proof}
In the notation of the paper \cite{SZ}, $W_r$ is isomorphic to the Lie algebra
$W(r,0,0,0,\{0\})$. Hence, it follows from \cite[Lemma 3.2]{SZ} that $\HCE^2
(W_r,\F)=0$. Moreover, as already has been observed in the proof of Proposition
\ref{complete}, in the notation of the paper \cite{SS}, $K_{2r+1}$ is isomorphic
to the Lie algebra $\mathscr{K}(\vec{\ell}, \sigma,\Gamma,\mathscr{I})$, where
$\vec{\ell}:=(0,0,0,0,0,r)$, $\sigma:=0$, $\Gamma:=\{0\}$, and $\mathscr{I}:=
\mathbb{N}_0^{1+2r}$. Therefore, by \cite[Lemma~4.1]{SS}, we conclude that
$\HCE^2(K_{2r+1},\F)=0$, as desired.  
\end{proof}


\section{Covering algebras}\label{covalg}


Let $0\to\bar{C}\to\bar{L}\to L\to 0$ be a central extension of a Lie algebra $L$. Then
$\bar{L}$ is called a \emph{covering algebra} of $L$ if $\bar{L}$ is perfect, i.e., $\bar{L}
=[\bar{L},\bar{L}]$ (see \cite[Definition 1.4]{G} or \cite[Section~2]{BM}). Since homomorphic
images of perfect Lie algebras are perfect, only perfect Lie algebras admit a covering algebra.
Moreover, if $C(L)=0$, then $\bar{C}=C(\bar{L})$.

\begin{lem}\label{split}
Let $L$ be a Lie algebra over a field $\F$ with $\HCE^2(L,\F)=0$. Then $L$ has at most
one covering algebra up to isomorphism.
\end{lem}

\begin{proof}
Let $\bar{L}$ be any covering algebra of $L$, i.e., $\bar{L}$ is perfect and there exists
a central extension $0\to\bar{C}\to\bar{L}\to L\to 0$. Since we have that $\HCE^2(L,
\F)=0$, the central extension splits, i.e., $\bar{L}\cong L\times\bar{C}$, and thus
we obtain that
$$\bar{L}=[\bar{L},\bar{L}]\cong[L\times\bar{C},L\times\bar{C}]=[L,L]=L\,,$$
which completes the proof.
\end{proof}

A \emph{universal central extension} of a Lie algebra $L$ is a central extension
$$0\to\hat{C}\to\hat{L}\stackrel{\hat{\pi}}\to L\to 0$$
of $L$ such that for every central extension $0\to Z\to E\stackrel{\pi}\to L\to 0$
of $L$ there exists a unique homomorphism $\varphi:\hat{L}\to E$ of Lie algebras
satisfying $\pi\circ\varphi=\hat{\pi}$. Note that then necessarily $\varphi(\hat{C})
\subseteq Z$. It is clear from this definition that a universal central extension of
$L$ is unique up to isomorphism, in case it exists. Since universal central extensions
must be perfect, they are covering algebras. It is well known that a Lie algebra $L$
over a field $\F$ has a universal central extension if, and only if, $L$ is perfect
(see \cite[Theorem 7.9.2]{W}). Moreover, we conclude from \cite[Lemma 3.1]{B1}
and \cite[Lemma 7.9.3]{W} that a central extension $0\to\hat{C}\to\hat{L}
\stackrel{\hat{\pi}}\to L\to 0$ is universal exactly when $\HCE^1(\hat{L},\F)=
\HCE^2(\hat{L},\F)=0$. In particular, the property of being universal is independent
of $L$. This will be used in the proof of the next result which shows that every
covering algebra of a Lie algebra $L$ is a homomorphic image of some universal
central extension of $L$. For this reason sometimes universal central extensions
are called {\em universal covering algebras\/} (e.g.\ cf.\ \cite[Definition 1.6]{G}
or \cite[Section 2]{BM}).

\begin{lem}\label{covalg}
Let $L$ be a perfect Lie algebra. Then a universal central extension of $L$ is
also a universal central extension of every covering algebra of $L$. In particular,
if $0\to\hat{C}\to\hat{L}\stackrel{\hat{\pi}}\to L\to 0$ is a universal central
extension of $L$, and $\bar{L}$ is a covering algebra of $L$, then there
exists a central ideal $C$ of $\hat{L}$ such that $C\subseteq\hat{C}$ and
$\bar{L}\cong\hat{L}/C$.
\end{lem}

\begin{proof}
Let $\bar{L}$ be a covering algebra of $L$, i.e., $\bar{L}$ is perfect and
there exists a central extension $0\to\bar{C}\to\bar{L}\stackrel{\bar{\pi}}\to
L\to 0$. Let $\varphi:\hat{L}\to L$ be the unique homomorphism of Lie algebras
such that the following diagram commutes:
\[
\xymatrix{
0\ar[r]& \hat{C}\ar[d]\ar[r] & \hat{L}\ar[d]^{\varphi}\ar[r]^{\hat{\pi}} & {L}\ar@{=}[d]\ar[r] & 0\\
0\ar[r] & \bar{C}\ar[r] & \bar{L}\ar[r]^{\bar{\pi}} & L\ar[r] & 0}
\]
It follows from the commutativity of the diagram that $\bar{\pi}(\im(\varphi))=
L=\bar{\pi}(\bar{L})$, and therefore we obtain that $\bar{L}=\im(\varphi)+
\bar{C}$. Since $\bar{L}$ is perfect and $\bar{C}$ is central in $\bar{L}$,
we conclude then that $\varphi$ is surjective. Moreover, we obtain from the
commutativity of the diagram that $C:=\Ker(\varphi)\subseteq\Ker(\hat{\pi})=
\hat{C}\subseteq C(\hat{L})$, and therefore it follows from the cohomological
characterization of universal central extensions observed above that $0\to C
\to\hat{L}\stackrel{\varphi}\to\bar{L}\to 0$ is a universal central extension of
$\bar{L}$. Moreover, the second part of the assertion is an immediate consequence. 
\end{proof}

Note that Lemma \ref{covalg} is the key result in the proof of Theorem \ref{simple}
below and in the proof of the next result which will be needed to deduce Corollaries
\ref{covalgreswitt} and \ref{covalgwitt} below. 

\begin{cor}\label{1dimcent}
Let $L$ be a Lie algebra over a field $\F$. If $L$ admits a universal central extension
$0\to\hat{C}\to\hat{L}\stackrel{\hat{\pi}}\to L\to 0$ such that $\dim_\F\hat{C}=1$,
then $L$ and $\hat{L}$ are the only covering algebras of $L$ up to isomorphism.
\end{cor}

\begin{proof}
Let $\bar{L}$ be a covering algebra of $L$. It follows from Lemma \ref{covalg} that there
exists a central ideal $C$ of $\hat{L}$ such that $C\subseteq\hat{C}$ and $\bar{L}\cong
\hat{L}/C$. If $C=0$, then $\bar{L}\cong\hat{L}$. If $C\ne 0$, then for dimension reasons
$C=\hat{C}$, and so we have that $\bar{L}\cong\hat{L}/C=\hat{L}/\hat{C}\cong L$. 
\end{proof}

In the remainder of this section we apply Corollary \ref{1dimcent} to determine the covering
algebras of some Witt algebras.

Let $\F$ be a field of prime characteristic $p>3$. Then it is well known that the Witt algebra
$W:=W(1,1):=\Der(\F[x]/(x^p))$ is simple\footnote{This is the first non-classical simple modular
Lie algebra defined by Ernst Witt in the 1930's.} and has a unique non-split central extension
(up to equivalence), the \emph{restricted Virasoro algebra} $V$ (see \cite[Korollar on p.\
78]{Str})\footnote{Note that $V$ is a restricted Lie algebra (see Corollary 2.9 in Chapter
2 of \cite{SF}). Moreover, it follows from \cite[Theorem 5.1]{B1} that every Zassenhaus
algebra $W(1,n)$ over a field of prime characteristic $p>3$ has a unique non-split central extension
(up to equivalence), the so-called \emph{modular Virasoro algebra}.}. Such a Lie algebra is
explicitely given by 
$$V:=\bigoplus_{n=-1}^{p-2}\F e_n\oplus\F z$$
with Lie brackets defined as follows:
\begin{eqnarray*}
[e_m,e_n]=
\left\{
\begin{array}{cc}
(n-m)e_{m+n}\ & \mbox{if }-1\le m+n\le p-2\\ 
\frac{1}{6}(n-1)n(n+1)z & \mbox{ if }m+n=p\\
0 & \mbox{otherwise}
\end{array}
\right.
\end{eqnarray*}
and $[e_n,z]=0$ for every $n\in\{-1,0,1,\dots,p-3,p-2\}$.

The next result is perhaps well known. But since we could not find a suitable reference, we
include its proof for the convenience of the reader.

\begin{lem}\label{resvirasoro}
Let $V$ denote the restricted Virasoro algebra over a field $\F$ of prime characteristic $p>3$.
Then $\HCE^1(V,\F)=\HCE^2(V,\F)=0$. In particular, the restricted Virasoro algebra is a universal
central extension of the Witt algebra.
\end{lem}

\begin{proof}
In order to prove that $\HCE^1(V,\F)=0$, we need to show that $V$ is perfect. We have
that $[e_2,e_{p-2}]=\frac{1}{6}(p-3)(p-2)(p-1)z$, and as $\frac{1}{6}(p-3)(p-2)(p-1)\ne
0$, we obtain that $z\in[V,V]$. In addition, $[e_0,e_n]=ne_n$ for every $n\in\{-1,\dots,p-2\}$
shows that $e_n\in[V,V]$ for every $n\ne 0$. Finally, we obtain from $[e_{-1},e_1]=2e_0$
that $e_0\in[V,V]$. Hence, we conclude that $V=[V,V]$.

Now we will prove that $\HCE^2(V,\F)=0$. We obtain from \cite[Theorem 7]{HS} the following
exact sequence:
$$\HCE^1(V,\F)\to\HCE^0(W,\HCE^1(\F z,\F))\stackrel{d_2^\prime}\to\HCE^2(W,\F)
\stackrel{l_2}\to\HCE^2(V,\F)\to\HCE^1(W,\HCE^1(\F z,\F))\,.$$
Since we have
$\HCE^1(\F z,\F)\cong\Hom_\F(\F z,\F)\cong\F$, we deduce that $\HCE^0(W,\HCE^1(\F z,\F))
\cong\F$ and $\HCE^1(W,\HCE^1(\F z,\F))\cong\HCE^1(W,\F)=0$. This in conjunction with
$\HCE^1(V,\F)=0$ yields the short exact sequence
$$0\to\F\stackrel{d_2^\prime}\to\HCE^2(W,\F)\stackrel{l_2}\to\HCE^2(V,\F)\to 0\,.$$
As $\dim_\F\HCE^2(W,\F)=1$, the injectivity of $d_2^\prime$ implies that $d_2^\prime$
is surjective. Hence, we obtain from the exactness of the sequence that $\Ker(l_2)=\im
(d_2^\prime)=\HCE^2(W,\F)$, and thus $l_2=0$. So we conclude that $\HCE^2(V,\F)=
\im(l_2)=0$. 

Finally, it follows from \cite[Lemma 3.1]{B1} and \cite[Lemma 7.9.3]{W} in conjunction with
$\HCE^1(V,\F)=\HCE^2(V,\F)=0$ that $V$ is a universal central extension of $W$.
\end{proof}

As an immediate consequence of Corollary \ref{1dimcent} and Lemma \ref{resvirasoro} we
obtain:

\begin{cor}\label{covalgreswitt}
Let $\F$ be a field of prime characteristic $p>3$. Then the Witt algebra and the restricted
Virasoro algebra are the only covering algebras of the Witt algebra up to isomorphism.
\end{cor}

For the remainder of this section let $\F$ be a field of characteristic zero. Then it is well known
that the infinite-dimensional two-sided Witt algebra\footnote{In Olivier Mathieu's paper the
Virasoro algebra is the two-sided Witt algebra $\wc$ (see \cite[(0.1.4)]{M}), but we follow
here the general custom to call the universal central extension $\vc$ of $\wc$ the Virasoro
algebra (see also footnote 3 in \cite{M}).} $\wc:=\Der(\F[t,t^{-1}])$ is simple (see \cite[Lemma
1.14]{M}) and has a unique non-split central extension (up to equivalence), the \emph{Virasoro
algebra} $\vc$ (see \cite[Section 1.3]{KRR}). Set
$$\vc:=\bigoplus_{n\in\Z}\F d_n\oplus\F c$$
with Lie brackets given by
$$[d_m,d_n]=(n-m)d_{m+n}+\delta_{m+n,0}\frac{1}{12}(n-1)n(n+1)c\,.$$

The first part of the next result is well-known (see \cite[Remark 6.4]{S} and \cite[Section 5.1]{ES}).
For the convenience of the reader we indicate a proof that is the precise analogue of the proof
of Lemma \ref{resvirasoro}.

\begin{lem}\label{virasoro}
Let $\vc$ denote the Virasoro algebra over a field $\F$ of characteristic zero. Then $\HCE^1
(\vc,\F)=\HCE^2(\vc,\F)=0$. In particular, the Virasoro algebra $\vc$ is a universal central
extension of the two-sided Witt algebra $\wc$.
\end{lem}

\begin{proof}
In order to prove that $\HCE^1(\vc,\F)=0$, we need to show that $\vc$ is perfect. We have
that $[d_0,d_n]=nd_n$ for any integer $n\ne 0$ which shows that $d_n\in[\vc,\vc]$ for $n
\ne 0$. Moreover, from $[d_{-1},d_1]=2d_0$ we obtain that $d_0\in[\vc,\vc]$. Finally,
$[d_{-2},d_2]=4d_0+\frac{1}{2}c$ implies that $c\in[\vc,\vc]$, and thus we have that
$\vc=[\vc,\vc]$.

Now we can use the same arguments as in the proof of Lemma \ref{resvirasoro} in conjunction
with $\HCE^1(\wc,\F)=0$ and $\dim_\F\HCE^2(\wc,\F)=1$ (see \cite[Section 2]{S}) to complete
the proof.
\end{proof}

Our final result of this section is then an immediate consequence of Corollary~\ref{1dimcent}
and Lemma \ref{virasoro}.

\begin{cor}\label{covalgwitt}
Let $\F$ be a field of characteristic zero. Then $\wc$ and the Virasoro algebra $\vc$
are the only covering algebras of the two-sided Witt algebra $\wc$ up to isomorphism.
\end{cor}


\section{Lie algebras whose derivation algebras are simple }\label{simderalg}


In this section we characterize Lie algebras whose derivation algebras are simple
or simple and complete. For a finite-dimensional Lie algebra $L$ over a field of
characteristic zero it follows from \cite[Theorem 4.4]{H} that $\Der(L)$ is simple
exactly when $L$ is simple. Some partial results for not necessarily finite-dimensional
Lie algebras over arbitrary fields were obtained by de Ruiter (see \cite[Theorem
3]{dR}). The next result settles this problem completely. 

We begin by introducing some notation. In analogy to the {\em idealizer\/} of a
subspace of $L$ we consider the {\em stabilizer\/} of a subspace $V$ of a (left)
$L$-module $M$, i.e.,
$$\stab_L(V):=\{x\in L\vert\,x\cdot V\subseteq V\}\,.$$  
Clearly, $\stab_L(V)$ is a subalgebra of $L$, and $V$ is an $L$-submodule of $M$
exactly when $\stab_L(V)=L$. Now, let
$$0\to\hat{C}\to\hat{\gf}\stackrel{\hat{\pi}}\to\gf\to 0$$
be a universal central extension of a centerless Lie algebra $\gf$. Then clearly $\hat{C}=
C(\hat{\gf})$ and, in view of \cite[Theorem 2.2]{BM}, every derivation $D$ of $\gf$ can
be uniquely lifted to a derivation $\hat{D}$ of $\hat{\gf}$, so that one has $\Der(\gf)
\cong\Der(\hat{\gf})$. Thus, we can view $C(\hat{\gf})$ as a $\Der(\gf)$-module by
setting $D\cdot x:=\hat{D}(x)$, for every $D\in\Der(L)$ and every $x\in C(\hat{\gf})$.
With respect to this action, $C(\hat{\gf})$ is annihilated by $\ad(\gf)$, so we can regard
$C(\hat{\gf})$ as an $\mathrm{Out}(\gf)$-module as well, where $\mathrm{Out}(\gf):=
\Der(\gf)/\ad(\gf)$ denotes the Lie algebra of outer derivations of $\gf$.

\begin{thm}\label{simple}
Let $L$ be a Lie algebra over a field $\F$. Then $\Der(L)$ is simple if, and only if, $L\cong
\hat{\gf}/C$, where $\hat{\gf}$ is a universal central extension of a simple Lie algebra
$\gf$ over $\F$ and $C$ is a central ideal of $\hat{\gf}$ such that $\stab_{\mathrm{Out}
(\gf)}(C)=0$.
\end{thm}

\begin{proof}
Suppose that $\Der(L)$ is simple. As $\ad(L)$ is an ideal of $\Der(L)$, we have $\Der
(L)=\ad(L)\cong L/C(L)$. Consequently, $L$ is a central extension of the simple Lie
algebra $\gf:=\Der(L)$. Moreover, it follows from \cite[Theorem 3\,(3)]{dR} that $L$
is perfect, and therefore $L$ is a covering algebra of $\gf$.

If $0\to\hat{C}\to\hat{\gf}\stackrel{\hat{\pi}}\to \gf\to 0$ is a universal central extension
of $\gf$, then we have that $\hat C=C(\hat{\gf})$ because $\gf$ is centerless. Thus, by
Lemma \ref{covalg}, there exists a central ideal $C$ of $\hat{\gf}$ such that $L\cong
\hat{\gf}/C$. Consequently, we deduce from \cite[Theorem 2.6\,(c)]{N} that $\Der(L)$
is isomorphic to the subalgebra of $\Der(\hat{\gf})$ consisting of all derivations $D$ such
that $D(C)\subseteq C$. Since this subalgebra contains $\ad(\hat{\gf})$ as an ideal, the
simplicity of $\Der(L)$ implies that the derivations of $\hat{\gf}$ leaving $C$ invariant are
precisely the inner derivations. This shows that $\stab_{\mathrm{Out}(\gf)}(C)=0$, which
finishes the proof that the given conditions are necessary.

Conversely, suppose that $0\to\hat{C}\to\hat{\gf}\stackrel{\hat{\pi}}\to\gf\to 0$ is a 
universal central extension of a simple Lie algebra $\gf$ such that $L\cong\hat{\gf}/C$
and $\stab_{\mathrm{Out}(\gf)}(C)=0$ for some central ideal $C$ of $\hat{\gf}$. Since
$\gf$ is centerless and $\hat{\gf}$ is perfect, we have that $\hat C=C(\hat{\gf})$ and
$L$ is perfect. From this we conclude that $L$ is a covering algebra of $\gf$. Hence, it
follows from Lemma \ref{covalg} that $\hat{\gf}$ is also a universal central extension
of $L$. As a consequence, it follows from \cite[Theorem 2.6\,(c)]{N} that $\Der(L)$
is isomorphic to the subalgebra of $\Der(\hat{\gf})$ consisting of all derivations $D$
such that $D(C)\subseteq C$. But $\stab_{\mathrm{Out}(\gf)}(C)=0$, so we necessarily
have $\Der(L)\cong\ad(\hat{\gf})\cong \hat{\gf}/C(\hat{\gf})\cong \gf$. In particular,
as $\gf$ is simple, $\Der(L)$ is simple.
\end{proof}

\begin{cor}
\label{simplecomplete}
Let $L$ be a Lie algebra over a field $\F$. Then $\Der(L)$ is simple and complete if,
and only if, $L$ is a covering algebra of a complete simple Lie algebra over $\F$.
\end{cor}

\begin{proof}
Firstly, suppose that $L$ is a covering algebra of a complete simple Lie algebra $\gf$.
Since $\gf$ is perfect, it follows from \cite[Theorem 7.9.2]{W} that $\gf$ has a universal
central extension $0\to\hat{C}\to\hat{\gf}\stackrel{\hat{\pi}}\to\gf\to 0$. By virtue
of Lemma \ref{covalg}, $\hat{\gf}$ is also a universal central extension of $L$. As
a consequence, we conclude from \cite[Theorem 2.2]{BM} in conjunction with the
completeness of $\gf$ that $\Der(L)\cong\Der(\hat{\gf})\cong\Der(\gf)\cong\gf$.
In particular, as $\gf$ is simple and complete, $\Der(L)$ is also simple and complete,
which shows the sufficiency of the condition.

Conversely, suppose now that $\Der(L)$ is simple and complete. By Theorem
\ref{simple} there exists a simple Lie algebra $\gf$ such that $L\cong\hat{\gf}/C$,
where $\hat{\gf}$ is the universal central extension of $\gf$ and $C$ is a central
ideal of $\hat{\gf}$. Then we have that
$$L/C(L)\cong[\hat{g}/C]/[C(\hat{g})/C]\cong\hat{\gf}/C(\hat{\gf})\cong\gf\,,$$
and because $\hat{g}$ is perfect, the same is true for $L$. Hence, $L$ is a covering
algebra of $\gf$. Finally, as in second part of the proof of Theorem \ref{simple}, we
obtain that $\gf\cong\Der(L)$, which implies that $\gf$ is simple and complete.
\end{proof}

Let $L$ be a finite-dimensional Lie algebra over a field of characteristic zero. From Theorem
\ref{simple} in conjunction with Whitehead's lemmas and Lemma \ref{split} it follows that
$\Der(L)$ is simple (and complete) if, and only if, $L$ is simple (and complete). According to
(v) of Corollary~\ref{primchar} below, the same conclusion fails in non-zero characteristic.
By using the Block-Wilson-Strade-Premet classification of finite-dimensional modular simple
Lie algebras (see \cite[Theorem 1.1]{PS}), we obtain the following result which shows that,
with the exception of the restricted Virasoro algebra, all finite-dimensional Lie algebras
over an algebraically closed field of prime characteristic $p>3$ that have a complete
simple derivation algebra are themselves simple and complete.

\begin{cor}\label{primchar}
Let $L$ be a finite-dimensional Lie algebra over an algebraically closed field $\F$ of prime
characteristic $p>3$. Then $\Der(L)$ is simple and complete if, and only if, $L$ is isomorphic
to one of the following Lie algebras:
\begin{enumerate}
\item[{\rm (i)}] a simple Lie algebra of classical type that is not isomorphic to $\psl_{np}(\F)$
                          for some positive integer $n$, or
\item[{\rm (ii)}] a Jacobson-Witt algebra $W(r,\underline{1})$, or
\item[{\rm (iii)}] a restricted contact Lie algebra $K(2r+1,\underline{1})^{(1)}$, where $r+2
                          \not\equiv 0\,\mbox{\rm (mod $p$)}$, or
\item[{\rm (iv)}] the restricted Melikian algebra $\mc(1,1)$, or
\item[{\rm (v)}] the restricted Virasoro algebra $V$.
\end{enumerate}
\end{cor}

\begin{proof}
In view of Corollary \ref{simplecomplete}, $\Der(L)$ is simple and complete if, and only
if, $L$ is a covering algebra of a complete simple Lie algebra $\gf$. According to
\cite[Theorem~1.1]{PS}, $\gf$ is either of classical or Cartan type or isomorphic to
a Melikian algebra. By virtue of Remark (1) after Theorem 1 in \cite{B}, the only simple
Lie algebra of classical type over a field of prime characteristic $p>3$ that has outer
derivations is $\psl_{pn}(\F)$ for some positive integer $n$. Moreover, we obtain
from \cite[Theorem 7.1.2]{StrI} that the only complete simple Lie algebras of Cartan
type are the Jacobson-Witt algebras $W(r,\underline{1})$ or those restricted contact
Lie algebras $K(2r+1,\underline{1})^{(1)}$ for which $r+2\not\equiv 0\,\mbox{\rm
(mod $p$)}$. Finally, it follows from \cite[Theorem 7.1.4]{StrI} that the only complete
Melikian algebra is $\mc(1,1)$.

It remains to determine the corresponding covering algebras of $\gf$. Suppose first that
$L$ is a covering algebra of a simple Lie algebra $\gf$ of classical type that is not isomorphic
to $\psl_{np}(\F)$ for some positive integer $n$. Then it follows from Theorem~3.1 of
\cite{B1} in conjunction with Lemma \ref{split} that $L\cong\gf$.

Next, suppose that $L$ is a covering algebra of $\gf$, where $\gf$ is either a Jacobson-Witt
algebra $W(r,\underline{1})$ with $r>1$ or a restricted contact Lie algebra $K(2r+1,
\underline{1})^{(1)}$ with $r+2\not\equiv 0\,\mbox{\rm (mod $p$)}$. Then we obtain
from \cite[Theorem 3.2\,(1) and Theorem 4.5]{F2}\footnote{Note that the results for
the central extensions of the contact Lie algebras announced in \cite[Theorem 2]{D1}
contain some typos.} that $\HCE^2(\gf,\F)=0$. Now as in the previous paragraph we
can apply Lemma \ref{split} to conclude that $L\cong\gf$.

Thirdly, suppose that $L$ is a covering algebra of the restricted Melikian algebra $\mc
(1,1)$. According to \cite[Proposition 6.2]{PS}, we have that $\HCE^2(\mc(1,1),\F)=0$,
and thus as before we can apply Lemma \ref{split} to deduce that $L\cong\mc(1,1)$.

Finally, suppose that $L$ is a covering algebra of the Witt algebra $W(1,1)$. Then we
obtain from Corollary \ref{covalgreswitt} that $L\cong W(1,1)$ or $L$ is isomorphic to
the restricted Virasoro algebra $V$.
\end{proof}

\begin{rem}\label{remprimchar}
\emph{We would like to point out to the interested reader the following observations:}

\emph{(1) Note that a complete Lie algebra is necessarily restricted with a unique
$[p]$-map (see the remark after the definition on p.\ 72 and Corollary 2.2\,(1) in
Chapter~2 of \cite{SF}). So it should not be surprising that all the Lie algebras
appearing in (i)--(v) of Corollary \ref{primchar} are restricted. Moreover, in principle,
in the proof of Corollary~\ref{primchar} we would have needed only the classification
of the \emph{restricted} simple Lie algebras (see \cite{BW}). But since the Block-Wilson
classification of the finite-dimensional restricted simple Lie algebras is only valid for
algebraically closed fields of prime characteristic $p>7$, this would have given a
weaker result.}

\emph{(2) One might be tempted to use Theorem \ref{simple} and the Block-Wilson-Strade-Premet
classification in order to find also an explicit description of those finite-dimensional Lie
algebras over algebraically closed fields of prime characteristic $p>3$ that have a (not
necessarily complete) simple derivation algebra. However, for Lie algebras $\gf$ of
Cartan type, it turns out to be hard to determine those central ideals $C$ that satisfy
the conditions in Theorem \ref{simple} unless $\mathrm{Out}(\gf)=0$ or $\dim_\F
\hat{C}\leq 1$. In this regard, it is worth recalling that $\hat{C}$ is isomorphic to
$\HCE_2(\gf,\F)$, the second homology space of $\gf$ with coefficients in the
trivial module $\F$ (see e.g., \cite[Theorem~7.9.2]{W}). In turn, it follows from
Proposition 5.1 in Chapter VI of \cite{CE} that $\HCE_2(\gf,\F)^*\cong\HCE^2
(\gf,\F)$, where $(-)^*$ denotes the linear dual. Consequently, the finite-dimensional
(not necessarily restricted) graded simple Lie algebras for which $\dim_\F\hat{C}>1$
holds are the special Lie algebras $S(r,\underline{n})^{(1)}$, the Hamiltonian algebras
$H(2r,\underline{n})^{(2)}$, and those contact Lie algebras $K(2r+1,\underline{n})^{(1)}$
that satisfy $r+2\equiv 0\,\mbox{\rm (mod $p$)}$ or $r+3\equiv 0\,\mbox{\rm (mod
$p$)}$ (see \cite[Corollary 4.1]{C}, \cite[Theorem 2.4\,(2) and (3)]{F1}, and
\cite[Theorem~4.5]{F2}). But in addition, there are also the non-graded simple
Lie algebras of Cartan type whose central extensions do not seem to be known
except in rank one (see \cite[Theorem~5.1]{B1} for the Albert-Frank algebras and
\cite[Theorem 6.3]{Str2} for the Block algebras).}

\emph{(3) As mentioned earlier, for a finite-dimensional Lie algebra $L$ over a field
of characteristic zero, Hochschild proved in \cite[Theorem 4.4]{H} that $\Der(L)$ is
semi-simple if, and only if, $L$ is semi-simple. In view of (v) in Corollary \ref{primchar},
this is not true in non-zero characteristic. But it would be very interesting to know
which (finite-dimensional) modular Lie algebras have a semi-simple derivation algebra.
Unfortunately, at the moment it seems that this problem is out of reach. Even in the
finite-dimensional case, the main difficulty is that modular semi-simple Lie algebras need
not be a direct product of simple Lie algebras, and the description provided by Richard
Block in \cite[Theorem 9.3]{B2} seems not to be helpful for our purposes.}

\emph{(4) The restricted Virasoro algebra $V$ is an example of a finite-dimensional
Lie algebra over a field of non-zero characteristic whose derivation algebra is simple, but
whose center is non-zero (see also the Remark after Theorem 3 in \cite{dR}, where an
example of an infinite-dimensional Lie algebras over a field of characteristic zero having
these properties is mentioned). Of course, by virtue of Hochschild's classical result
\cite[Theorem 4.4]{H}, every finite-dimensional Lie algebra over a field of characteristic
zero whose derivation algebra is simple is itself semi-simple, and thus its center is zero.}
\end{rem}

As another application of Corollary \ref{simplecomplete} we determine those Lie algebras
that have a complete simple derivation algebra and are $\Z$-graded of finite growth over
an algebraically closed field $\F$ of characteristic zero.

\begin{cor}\label{char0}
Let $L$ be a $\Z$-graded Lie algebra of finite growth over an algebraically closed field $\F$
of characteristic zero. Then $\Der(L)$ is simple and complete if, and only if, $L$ is isomorphic
to one of the following Lie algebras:
\begin{enumerate}
\item[{\rm (i)}] a finite-dimensional simple Lie algebra, or
\item[{\rm (ii)}] a one-sided Witt algebra $W_r$, or
\item[{\rm (iii)}] a contact Lie algebra $K_{2r+1}$, or
\item[{\rm (iv)}] the two-sided Witt algebra $\wc$, or
\item[{\rm (v)}] the Virasoro algebra $\vc$.
\end{enumerate}
\end{cor}

\begin{proof}
According to Corollary \ref{simplecomplete}, $\Der(L)$ is simple and complete if, and only if,
$L$ is a covering algebra of a complete simple Lie algebra $\gf$. It follows from the main
theorem and Lemma 1.14 of \cite{M} that $\gf$ is either a finite-dimensional simple Lie
algebra or a Cartan algebra or it is isomorphic to the two-sided Witt algebra $\wc$.

According to Whitehead's first lemma, every finite-dimensional simple Lie algebra is complete.
Moreover, it follows from Proposition \ref{complete} that $W_r$ and $K_{2r+1}$ are the
only complete Cartan algebras. Finally, it is known that the two-sided Witt algebra $\wc$
is complete (see \cite{ZM} or \cite[Theorem A.1.1]{ES}).

It remains to determine the corresponding covering algebras of $\gf$. Suppose first that
$L$ is a covering algebra of a finite-dimensional simple Lie algebra $\gf$. Then it follows
from Whitehead's second lemma in conjunction with Lemma \ref{split} that $L\cong\gf$.

Next, suppose that $L$ is a covering algebra of $\gf$, where $\gf$ is either a one-sided
Witt algebra $W_r$ or a contact Lie algebra $K_{2r+1}$. Then we obtain from Proposition~\ref{centext}
that $\HCE^2(\gf,\F)=0$. Now as in the previous paragraph we can apply Lemma \ref{split}
to conclude that $L\cong\gf$.

Finally, suppose that $L$ is a covering algebra of the two-sided Witt algebra $\wc$. Then
we obtain from Corollary \ref{covalgwitt} that $L\cong\wc$ or $L\cong\vc$.
\end{proof}

\begin{rem}
\emph{(1) Similar to Remark \ref{remprimchar}\,(2), one might be tempted to use
Theorem~\ref{simple} in conjunction with the main result of \cite{M} in order to find
also an explicit description of those $\Z$-graded Lie algebras of finite growth over an
algebraically closed field $\F$ of characteristic zero that have a (not necessarily complete)
simple derivation algebra. However, for Cartan algebras $\gf$, it turns out to be hard to
determine those central ideals $C$ that satisfy the conditions in Theorem~\ref{simple}
unless $\mathrm{Out}(\gf)=0$ or $\dim_\F\hat{C}\leq 1$. As already has been
observed in Remark \ref{remprimchar}\,(2), we have that $\hat{C}\cong\HCE_2
(\gf,\F)$ and $\HCE_2(\gf,\F)^*\cong\HCE^2 (\gf,\F)$, and therefore Theorem 1 in
\cite{D2} indicates that the $\Z$-graded simple Lie algebras for which $\dim_\F
\hat{C}>1$ seem to be again the special Lie algebras and the Hamiltonian algebras.}

\emph{(2) The Virasoro algebra $\vc$ is an example of an infinite-dimensional
Lie algebra whose derivation algebra is simple, but whose center is non-zero (see the
Remark after Theorem 3 in \cite{dR} for another such example due to de la Harpe that
seems to be much more complicated).}
\end{rem}

It should be mentioned that each of the Lie algebras appearing in (i)--(iii) and (v) of
Corollaries \ref{primchar} and \ref{char0} correspond remarkably well. Moreover, note
that the restricted Melikian algebra $\mc(1,1)$ in (iv) of Corollary \ref{primchar} is
closely related to the Jacobson-Witt algebra $W(2,(1,1))$, but unlike the two-sided Witt
algebra $\wc$, $\mc(1,1)$ has no non-trivial central extensions (see Proposition 6.2
in \cite{PS}). On the other hand, the one-sided Witt algebras $W_r$ have no non-trivial
central extensions (see Proposition \ref{centext}), but similar to the two-sided Witt
algebra $\wc$, the restricted Witt algebra $W(1,1)$ has a universal central extension,
namely, the restricted Virasoro algebra $V$ (see Lemma \ref{resvirasoro}).



\end{document}